\documentclass{article}

\usepackage{graphicx}
\usepackage{multicol}
\usepackage[bottom]{footmisc}
\usepackage{amsmath,amssymb,amsfonts, amsthm}
\usepackage[affil-it]{authblk}

\newcommand{\Z}{\mathbb Z}
\newcommand{\N}{\mathbb N}

\newcommand{\su}{\subseteq}

\newtheorem{thm}{Theorem}[section]
\newtheorem{lem}[thm]{Lemma}
\newtheorem{cor}[thm]{Corollary}
\newtheorem{obs}[thm]{Observation}
\newtheorem{prop}[thm]{Proposition}

\theoremstyle{definition}
\newtheorem{defn}[thm]{Definition}
\newtheorem{rem}[thm]{Remark}

\theoremstyle{remark}
\newtheorem{case}{Case}[thm]

\author{Santanu Mondal, Krishnendu Paul, Shameek Paul
\thanks{E-mail addresses: \texttt{santanu.mondal.math18@gm.rkmvu.ac.in, krishnendu.p.math18@gm.rkmvu.ac.in, shameek.paul@rkmvu.ac.in}}}
\affil{\small Ramakrishna Mission Vivekananda Educational and Research Institute, Belur, Dist. Howrah, 711202, India}

\date{}

\begin{document}
\baselineskip=14.5pt
\title {Zero-sum constants related to the Jacobi symbol} 

\maketitle

\begin{abstract}
For a weight-set $A\su \Z_n$, the $A$-weighted Davenport constant $D_A(n)$ is defined to be the smallest natural number $k$ such that any sequence of $k$ elements in $\Z_n$ has an $A$-weighted zero-sum subsequence and the constant $C_A(n)$ is defined to be the smallest natural number $k$ such that any sequence of $k$ elements in $\Z_n$ has an $A$-weighted zero-sum subsequence of consecutive terms. 

When $n$ is odd, for $x\in \Z_n$ let $\big(\frac{x}{n}\big)$ be the Jacobi symbol and $S(n)=\big\{\,x\in U(n):\big(\frac{x}{n}\big)=1\,\big\}$. We compute these constants for the weight-set $S(n)$. For a prime divisor $p$ of $n$, we also compute these constants for the weight-set $L(n;p)=\big\{\,x\in U(n):\big(\frac{x}{n}\big)=\big(\frac{x}{p}\big)\,\big\}$. We show that even though these weight-sets may have half the size of $U(n)$, they can have the same constants as for $U(n)$. 
\end{abstract}

\bigskip


\bigskip


\vspace{.3cm}

\section{Introduction}\label{0}

The following definition was given in \cite{AR}.

\begin{defn} 
For a weight-set $A \su \Z_n$, the $A$-weighted Davenport constant $D_A(n)$ is defined to be the least positive integer $k$, such that any sequence in $\Z_n$ of length $k$ has an $A$-weighted zero-sum subsequence. 
\end{defn}

The following definition was given in \cite{SKS1}.

\begin{defn}
For a weight-set $A \su \Z_n$, the $A$-weighted constant $C_A(n)$ is defined to be the least positive integer $k$, such that any sequence in $\Z_n$ of length $k$ has an $A$-weighted zero-sum subsequence of consecutive terms. 
\end{defn}

Let $U(n)$ denote the multiplicative group of units in the ring $\Z_n$, and let $U(n)^2=\{\,x^2:x\in U(n)\,\}$. For an odd prime $p$, let $Q_p$ denote the set $U(p)^2$. For $n$ squarefree, let $\Omega(n)$ denote the number of distinct prime divisors of $n$. 

Let $m$ be a divisor of $n$. We refer to the ring homomorphism $f_{n|m}:\Z_n\to \Z_m$ given by $a+n\Z\mapsto a+m\Z$ as the natural map. As this map sends units to units, we get a group homomorphism $U(n)\to U(m)$ which we also refer to as the natural map.

The Jacobi symbol is defined in Section \ref{s}, for odd $n$. It is denoted by $\big(\frac{x}{n}\big)$. The following are some of the results in this paper. Except for the first result, we assume that $n$ is an odd, squarefree number whose every prime divisor is at least seven. 

\begin{itemize}
\item 
Let $S(n)=\big\{\,x\in U(n):\big(\frac{x}{n}\big)=1\,\big\}$. 
 
If $n$ is prime, then $D_{S(n)}(n)=3$, and $D_{S(n)}(n)=\Omega(n)+1$ otherwise. 
 
If $n$ is prime, then $C_{S(n)}(n)=3$, and $C_{S(n)}(n)=2^{\Omega(n)}$ otherwise.

\item 
Let $L(n;p)=\big\{\,x\in U(n) : \big(\frac{x}{n}\big)=\big(\frac{x}{p}\big)\,\big\}$ where $p$ is a prime divisor of $n$. 
 
If $\Omega(n)=2$, then $D_{L(n;p)}(n)=4$, and $D_{L(n;p)}(n)=\Omega(n)+1$ otherwise.

If $\Omega(n)=2$, then $C_{L(n;p)}(n)=6$, and $C_{L(n;p)}(n)=2^{\Omega(n)}$ otherwise.
  
\end{itemize}

\begin{rem}
In \cite{AH} it was shown that if $A=\Z_n\setminus\{0\}$ and $B=\{1,2,\ldots,\lceil n/2\rceil\}$, we have $D_A(n)=D_B(n)$. We make a similar observation in this paper. In Proposition \ref{indexs}, we show that $S(n)$ is a subgroup of $U(n)$ having index two when $n$ is not a square. From \cite{sg} and \cite{SKS1} we see that when $n$ is odd, we have $D_{U(n)}(n)=\Omega(n)+1$ and $C_{U(n)}(n)=2^{\Omega(n)}$. So from Theorems \ref{dsn} and \ref{csn}, we see that if in addition $n$ is not a prime, we have $D_{S(n)}(n)=D_{U(n)}(n)$ and $C_{S(n)}(n)=C_{U(n)}(n)$. Thus, even though these weight-sets may have different sizes, they can have the same constants. From Theorems \ref{ld} and \ref{lc}, we see that when $\Omega(n)\neq 2$, we have $D_{L(n;p)}(n)=D_{U(n)}(n)$ and $C_{L(n;p)}(n)=C_{U(n)}(n)$.  . 
\end{rem}

If $p$ is a prime divisor of $n$, we use the notation $v_p(n)=r$ to mean that $p^r\mid n$ and $p^{r+1}\nmid n$.    
Let $p$ be a prime divisor of $n$ and $v_p(n)=r$. We denote the image in $U(p^r)$ of $x\in U(n)$ under $f_{n|p^r}$ by $x^{(p)}$. For a sequence $S=(x_1,\,\ldots,\,x_l)$ in $\Z_n$, let $S^{(p)}$ denote the sequence $(x_1^{(p)},\,\ldots,\,x_l^{(p)})$ in $\Z_{p^r}$ which is the image of $S$ under the $f_{n|p^r}$. The following statement is Observation 2.2 in \cite{sg}. 

\begin{obs}\label{obs2}
Let $S$ be a sequence in $\Z_n$. Suppose for every prime divisor $p$ of $n$, the sequence $S^{(p)}$ in $\Z_{p^r}$ is a $U(p^r)$-weighted zero-sum sequence where $r=v_p(n)$. Then $S$ is a $U(n)$-weighted zero-sum sequence.  
\end{obs}

The next result follows from Theorem 1.2 of \cite{YZ} along with Theorem 1 of \cite{L}, and from Corollary 4 of \cite{SKS1}.

\begin{thm}\label{un}
Let $n$ be odd. Then $D_{U(n)}(n)=\Omega(n)+1$ and $C_{U(n)}(n)=2^{\Omega(n)}$. 
\end{thm}

We get the next result from Theorem 2 of \cite{AR} and Theorem 4 of \cite{SKS1}. 

\begin{thm}\label{q}
Let $p$ be an odd prime. Then $C_{Q_p}(p)=D_{Q_p}(p)=3$. 
\end{thm}

The next result is Lemma 3 of \cite{SKS1} which will be used in Theorem \ref{lc2}. 

\begin{lem}\label{clb}
Let $n=mq$. Let $A,B,C$ be subsets of $\Z_n,\Z_m,\Z_q$ respectively. Suppose $f_{n|m}(A)\su B$ and $f_{n|q}(A)\su C$. Then we have $C_A(n)\geq C_B(m)\,C_C(q)$. 
\end{lem} 

We now prove a similar result for the weighted Davenport constant which we will use in Theorem \ref{ld2}. A generalization of this result was proved in Lemma 3.1 of \cite{grn} for abelian groups.

\begin{lem}\label{dadd}
Let $n=mq$. Let $A,B,C$ be subsets of $\Z_n,\Z_m,\Z_q$ respectively. Suppose $f_{n|m}(A)\su B$ and $f_{n|q}(A)\su C$. Then $D_A(n)\geq D_B(m)+D_C(q)-1$. 
\end{lem} 

\begin{proof}
Let $D_B(m)=k$ and $D_C(q)=l$. Assume that $k,l\geq 2$.  There exists a sequence $S_1'=(u_1,\ldots,u_{k-1})$ of length $k-1$ in $\Z_m$ which has no $B$-weighted zero-sum subsequence, and there exists a sequence $S_2'=(v_1,\ldots,v_{l-1})$ of length $l-1$ in $\Z_q$ which has no $C$-weighted zero-sum subsequence. 

As $f_{n|m}$ is onto, for $1\leq i\leq k-1$ there exist $x_i\in \Z_n$ such that $f_{n|m}(x_i)=u_i$ and as $f_{n|q}$ is onto, for $1\leq j\leq l-1$ there exist $y_j\in\Z_n$ such that $f_{n|q}(y_j)=v_j$. Let $S$ be the sequence $(x_1q,\,\ldots,\,x_{k-1}q,\,y_1,\,\ldots,\,y_{\,l-1})$ of length $k+l-2$ in $\Z_n$. 

Let $S_1=(x_1q,\ldots,x_{k-1}q)$ and $S_2=(y_1,\ldots,y_{\,l-1})$. Suppose $S$ has an $A$-weighted zero-sum subsequence $T$. If the sequence $T$ contains some term of $S_2$, by taking the image of $T$ under $f_{n|q}$ we get the contradiction that $S_2'$ has a $C$-weighted zero-sum subsequence, as $f_{n|q}(x_iq)=0$ and as $f_{n|q}(A)\su C$. 

Thus, $T$ does not contain any term of $S_2$ and so $T$ is a subsequence of $S_1$. Let $T'$ be the subsequence of $S_1'$ such that $u_i$ is a term of $T'$ if and only if $x_iq$ is a term of $T$. As $f_{n|m}(A)\su B$, by dividing the $A$-weighted zero-sum which is obtained from $T$ by $q$ and by taking the image under $f_{n|m}$ we get the contradiction that $T'$ is a $B$-weighted zero-sum subsequence of $S_1'$. 

Hence, we see that $S$ does not have any $A$-weighted zero-sum subsequence. As $S$ has length $k+l-2$, it follows that $D_A(n)\geq k+l-1$.

If $k=l=1$, then we are done. Suppose exactly one of them is equal to one. We may assume that $k>1$ and $l=1$. Then we take $S_2'$ to be the empty sequence in the above proof and so $S_1=S$.
\end{proof}

\section{Some results about the weight-set $S(n)$}\label{s}

From this point onwards, we will always assume that $n$ is odd. 

\begin{defn}\label{jacobi}
For an odd prime $p$ and for $a\in U(p)$ the symbol $\Big(\dfrac{a}{p}\Big)$ is the Legendre symbol which is defined as $\Big(\dfrac{a}{p}\Big)=\begin{cases}
  ~1 & \mbox{ if $ a\in Q_p $}\\
  -1 & \mbox{ if $ a\notin Q_p $}
\end{cases}$ . 

For a prime divisor $p$ of $n$, we use the notation $\Big(\dfrac{a}{p}\Big)$ to denote $\bigg(\dfrac{f_{n|p}(a)}{p}\bigg)$ where $a\in U(n)$. Let $n=p_1^{r_1}\ldots p_k^{r_k}$ where the $p_i$'s are distinct primes. 

\smallskip

For $a\in U(n)$, we define the Jacobi symbol $\Big(\dfrac{a}{n}\Big)$ to be $\Big(\dfrac{a}{p_1}\Big)^{r_1}\ldots \Big(\dfrac{a}{p_k}\Big)^{r_k}$. Observe that we have 
$\Big(\dfrac{a}{n}\Big)=\bigg(\dfrac{a^{(p_1)}}{p_1^{r_1}}\bigg)\ldots \bigg(\dfrac{a^{(p_k)}}{p_k^{r_k}}\bigg)$. Let $S(n)$ denote the kernel of the homomorphism $U(n)\to\{1,-1\}$ given by $a\mapsto \Big(\dfrac{a}{n}\Big)$. 
\end{defn}

In Section 3 of \cite{ADU}, the set $S(n)$ was considered as a weight-set. 

\begin{prop}\label{indexs}
$S(n)$ is a subgroup having index two in $U(n)$ when $n$ is a non-square, and $S(n)=U(n)$ when $n$ is a square.  
\end{prop}

\begin{proof}
Let $n=p_1^{r_1}\ldots p_k^{r_k}$ where the $p_i$'s are distinct primes. If $n$ is a square, then all the $r_i$ are even and so $S(n)=U(n)$. If $n$ is not a square, there exists $j$ such that $r_j$ is odd. As for any $p,r\in\N$ the map $f_{p^r|p}$ is onto, by the Chinese Remainder theorem we see that there is a unit $b\in U(n)$ such that $\Big(\dfrac{b}{p_i}\Big)=1$ when $i\neq j$, and $\Big(\dfrac{b}{p_j}\Big)=-1$. It follows that $\Big(\dfrac{b}{n}\Big)=-1$ and so the homomorphism $U(n)\to\{1,-1\}$ given by $a\mapsto \Big(\dfrac{a}{n}\Big)$ is onto. Hence, we see that $S(n)$ has index two in $U(n)$. 
\end{proof}

\begin{rem}
In particular, if $n$ is squarefree then $S(n)$ has index two in $U(n)$. It follows that when $p$ is an odd prime we have $S(p)=Q_p$. 
\end{rem}

\begin{obs}\label{obs}
Let $n=p_1\ldots p_k$ where the $p_i$'s are distinct prime numbers. For $a\in U(n)$, let $\mu(a)$ denote the cardinality of $\{\,1\leq j\leq k:f_{n|p_j}(a)\notin Q_{p_j}\,\}$. Then $a\in S(n)$ if and only if $\mu(a)$ is even.  
\end{obs}

\begin{lem}\label{u2s}
Let $d$ be a proper divisor of $n$ such that $d$ is not a square. Suppose $d$ is coprime with $n'$ where $n'=n/d$. Then $U(n')\su f_{n|n'}(S(n))$.
\end{lem}

\begin{proof}
Let $a'\in U(n')$. By the Chinese remainder theorem, there is an isomorphism $\psi:U(n)\to U(n')\times U(d)$. If $a'\in S(n')$, let $a\in U(n)$ such that $\psi(a)=(a',1)$. If $a'\notin S(n')$, let $b\in U(d)\setminus S(d)$ and let $a\in U(n)$ such that $\psi(a)=(a',b)$. Such a $b$ exists by Proposition \ref{indexs} because $d$ is not a square. Then $a\in S(n)$ and $f_{n|n'}(a)=a'$. 
\end{proof}

\begin{lem}\label{lifts'}
Let $S$ be a sequence in $\Z_n$ and let $d$ be a proper divisor of $n$ which divides every element of $S$. Let $n'=n/d$ and let $d$ be coprime with $n'$. Let $S'$ be the sequence in $\Z_{n'}$ which is the image of the sequence $S$ under $f_{n|n'}$. Let $A\su\Z_n$ and let $A'\su \Z_{n'}$ such that $A'\su f_{n|n'}(A)$. Suppose $S'$ is an $A'$-weighted zero-sum sequence. Then $S$ is an $A$-weighted zero-sum sequence.  
\end{lem}

\begin{proof}
Let $S=(x_1,\ldots,x_k)$ be a sequence in $\Z_n$ and let $S'=(x_1',\ldots,x_k')$ where $x_i'=f_{n|n'}(x_i)$ for $1\leq i\leq k$. Suppose $S'$ is an $A'$-weighted zero-sum sequence. Then for any $1\leq i\leq k$, there exist $a_i'\in A'$ such that $a_1'x_1'+\cdots+a_k'x_k'=0$. Since  $A'\su f_{n|n'}(A)$, for $1\leq i\leq k$ there exist $a_i\in A$ such that $f_{n|n'}(a_i)=a_i'$. As $a_1'x_1'+\cdots+a_k'x_k'=0$ in $\Z_{n'}$, it follows that $f_{n|n'}(a_1x_1+\cdots+a_kx_k)=0$. Let $x=a_1x_1+\cdots+a_kx_k\in \Z_n$. As $f_{n|n'}(x)=0$, we see that $n'\mid x$ and as every term of $S$ is divisible by $d$, we see that $d\mid x$. Now as $d$ is coprime with $n'$, it follows that $x$ is divisible by $n=n'd$ and so $x=0$. Thus, $S$ is an $A$-weighted zero-sum sequence. 
\end{proof}

The next result is Lemma 2.1 (ii) of \cite{sg}, which we restate here using our terminology.

\begin{lem}\label{gri}
Let $p$ be an odd prime. If a sequence $S$ in $\Z_{p^r}$ has at least two terms coprime to $p$, then $S$ is a $U(p^r)$-weighted zero-sum sequence. 
\end{lem}

The next result is Lemma 1 in \cite{CM}. 

\begin{lem}\label{cm'}
Let $A=U(n)^2$ where $n=p^r$ and $p\geq 7$ is a prime. Let $x_1,x_2,x_3\in U(n)$. Then $Ax_1+Ax_2+Ax_3=\Z_n$. 
\end{lem}

For the theorem in the next section, we need the following lemma which is similar to Lemma \ref{cm'}. We observe that when $n=p^r$ where $p$ is an odd prime and $r\in\N$, then $U(n)$ is a cyclic group (see \cite{IR}) and so -1 is the unique element in $U(n)$ of order 2. Thus, the map $U(n)\to U(n)$ given by $x\mapsto x^2$ has  kernel $\{1,-1\}$ and so $U(n)^2$ is a subgroup of $U(n)$ having index 2. Hence, $|A_1|=|A_2|$ in the next lemma and so its proof is similar to the proof of Lemma 1 of \cite{CM}.

\begin{lem}\label{kneser}
Let $A_1=U(n)^2$ and $A_2=U(n)\setminus U(n)^2$, where $n=p^r$ and $p\geq 7$ is a prime. Let $x_1,x_2,x_3\in U(n)$ and let $f:\{1,2,3\}\to\{1,2\}$ be any function. Then $A_{_{f(1)}}x_1+A_{_{f(2)}}x_2+A_{_{f(3)}}x_3=\Z_n$. 
\end{lem}

\begin{cor}\label{cm}
Let $n=p^r$ and $p\geq 7$ be a prime. Let $S$ be a sequence in $\Z_n$ such that at least three terms of $S$ are in $U(n)$. Then $S$ is a $U(n)^2$-weighted zero-sum sequence. 
\end{cor}

\begin{proof}
Let $S=(x_1,x_2,\ldots,x_k)$ be a sequence in $\Z_n$ as in the statement of the corollary and $A=U(n)^2$. Without loss of generality, we may assume that $x_1,x_2,x_3\in U(n)$. If $k=3$, let $y=0$. If $k>3$, let $y=x_4+\cdots+x_k$. By Lemma \ref{cm'}, we get $-y\in Ax_1+Ax_2+Ax_3$. So there exists $a_1,a_2,a_3\in A$ such that $a_1x_1+a_2x_2+a_3x_3+y=0$. Thus, $S$ is an $A$-weighted zero-sum sequence. 
\end{proof}

\begin{rem}
The conclusion of Corollary \ref{cm} may not hold when $p<7$. One can check that the sequence $(1,1,1)$ in $\Z_n$ is not a $U(n)^2$-weighted zero-sum sequence, when $n$ is 2 or 5 and the sequence $(1,2,1)$ in $\Z_3$ is not a $U(3)^2$-weighted zero-sum sequence. 
\end{rem}

\section{The constants $D_{S(n)}(n)$ and $C_{S(n)}(n)$}

\begin{lem}\label{gs}
Let $n$ be squarefree and $S=(x_1,\ldots,x_l)$ be a sequence in $\Z_n$. Suppose given any prime divisor $p$ of $n$, at least two terms of $S$ are coprime to $p$. If at most one term of $S$ is a unit, then $S$ is an $S(n)$-weighted zero-sum sequence. 
\end{lem}

\begin{proof}
As we have assumed that $n$ is odd and as for every prime divisor $p$ of $n$ at least two terms of $S$ are coprime to $p$, by Lemma \ref{gri} for every prime divisor $p$ of $n$ the sequence $S^{(p)}=(x_1^{(p)},\ldots,x_l^{(p)})$ is a $U(p)$-weighted zero-sum sequence. Let  $n=p_1\ldots p_k$ where the $p_i$'s are distinct primes. For $1\leq i\leq k$ there exist $c_{i,1},\ldots,c_{i,l}\in U(p_i)$ such that $c_{i,1}x_1^{(p_i)}+\cdots +c_{i,l}x_l^{(p_i)}=0$. 

By Observation \ref{obs2}, for $1\leq j\leq l$ there exist $a_j\in U(n)$ such that $a_1x_1+\ldots+a_lx_l=0$ and such that for $1\leq i\leq k$ and we have $(a_1^{(p_i)},\ldots,a_l^{(p_i)})=(c_{i,1},\ldots,c_{i,l})$. Let $X$ denote the $k\times l$ matrix whose $i^{th}$ row is $(x_1^{(p_i)},\ldots,x_l^{(p_i)})$ and let $C$ denote the $k\times l$ matrix whose $i^{th}$ row is $(c_{i,1},\ldots,c_{i,l})$. We want to modify the entries of the matrix $C$ so that for $1\leq j\leq l$ the corresponding $a_j\in U(n)$ which we get by the Chinese remainder theorem are in $S(n)$.

Suppose the $j^{th}$ column of $X$ has a zero. Then there exists $1\leq i\leq k$ such that $x_j^{(p_i)}=0$. By making a suitable choice for $c_{i,j}$ we can ensure that the corresponding $a_j\in U(n)$ is in $S(n)$ as $\Big(\dfrac{a_j}{n}\Big)=\Big(\dfrac{c_{_{1,j}}}{p_{_1}}\Big)\ldots \Big(\dfrac{c_{_{k,j}}}{p_{_{k}}}\Big)$. Thus, we can modify the $j^{th}$ column of $C$ so that the corresponding $a_j\in U(n)$ is in $S(n)$. 

We observe that a term $x_j$ of $S$ is a unit if and only if the $j^{th}$ column of $X$ does not have a zero. Hence, if no term of $S$ is a unit then each column of $X$ has a zero. So in this case $S$ is an $S(n)$-weighted zero-sum sequence. 

Suppose exactly one term of $S$ is a unit, say $x_{j_0}$. Then the $j_0^{th}$ column of $X$ does not have a zero and there is a zero in all the other columns of $X$. By multiplying the $1^{st}$ row of $C$ by a suitable element of $U(p_1)$, we can modify $c_{1,j_0}$ so that $a_{j_0}\in S(n)$. As the other columns of $X$ have a zero, we can modify those columns of $C$ suitably so that $a_j\in S(n)$ for $j\neq j_0$. Thus, $S$ is an $S(n)$-weighted zero-sum sequence. 
\end{proof}

\begin{lem}\label{gs'}
Let $n$ be squarefree, every prime divisor of $n$ be at least seven and $S=(x_1,\ldots,x_l)$ be a sequence in $\Z_n$ such that for every prime divisor of $n$, at least two terms of $S$ are coprime to it. Suppose there is a prime divisor $p$ of $n$ such that at least three terms of $S$ are coprime to $p$. Then $S$ is an $S(n)$-weighted zero-sum sequence.
\end{lem}

\begin{proof}
If $\Omega(n)=1$, then $n$ is a prime say $p$. As at least three terms of $S$ are coprime to $p$, so by Corollary \ref{cm} we have $S$ is a $Q_p$-weighted zero-sum sequence.

Let $\Omega(n)\geq 2$. As there are at least three units in the sequence $S^{(p)}$, by Lemma \ref{gri} it is a $U(p)$-weighted zero-sum sequence. So for $1\leq i\leq l$ there exist $b_i\in U(p)$ such that $b_1x_1^{(p)}+\cdots+b_lx_l^{(p)}=0$. Let us assume that $x_1^{(p)},x_2^{(p)}$ and $x_3^{(p)}$ are units. A similar argument will work in the general case. We want to choose the $b_i$'s so that the corresponding $U(n)$-weighted zero-sum for $S$ (which we get using Observation \ref{obs2}, as in Lemma \ref{gs}) is an $S(n)$-weighted zero-sum. 

Using Observation \ref{obs} we choose the units $\{\,b_i:4\leq i\leq l\,\}$ so that for $4\leq i\leq l$ we have $a_i\in S(n)$.
Let us denote the negative of $b_4x_4^{(p)}+\cdots+b_lx_l^{(p)}$ by $y$. By Lemma \ref{kneser} and using Observation \ref{obs} we can choose $b_1,b_2,b_3\in U(p)$ so that $a_1,a_2,a_3\in S(n)$ and $b_1x_1^{(p)}+b_2x_2^{(p)}+b_3x_3^{(p)}=y$. Thus, $S$ is an $S(n)$-weighted zero-sum sequence.
\end{proof}

\begin{thm}\label{dsn}
Let $n$ be squarefree. If $n$ is prime, we have $D_{S(n)}(n)=3$. If $n$ is not a prime and every prime divisor of $n$ is at least seven, we have $D_{S(n)}(n)=\Omega(n)+1$. 
\end{thm}

\begin{proof}
From Theorem \ref{un} we have $D_{U(n)}(n)=\Omega(n)+1$. As $S(n)\su {U(n)}$ it follows that $D_{S(n)}(n)\geq D_{U(n)}(n)$ and so $D_{S(n)}(n)\geq \Omega(n)+1$. If $\Omega(n)=1$ then $n=p$ where $p$ is a prime and $S(n)=Q_p$. So by Theorem \ref{q}, we have $D_{S(n)}(n)=3$. 

Let $n$ be squarefree and let $\Omega(n)\geq 2$. We claim that $D_{S(n)}(n)\leq \Omega(n)+1$. Let $S=(x_1,\ldots,x_l)$ be a sequence in $\Z_n$ of length $l=k+1$ where $k=\Omega(n)$. We have to show that $S$ has an ${S(n)}$-weighted zero-sum subsequence. If any term of $S$ is zero, then that term will give us an ${S(n)}$-weighted zero-sum subsequence of length 1.  

\begin{case}
There is a prime divisor $p$ of $n$ such that at most one term of $S$ is coprime to $p$.
\end{case}

Let us assume without loss of generality that $x_i$ is divisible by $p$ for $2\leq i\leq l$ and let $T$ denote the subsequence $(x_2,\ldots,x_l)$ of $S$. Let $n'=n/p$ and let $T'$ be the sequence in $\Z_{n'}$ which is the image of $T$ under $f_{n|n'}$. From Theorem \ref{un}, we see that $D_{U(n')}(n')=\Omega(n')+1$. As $T'$ has length $l-1=\Omega(n)=\Omega(n')+1$, it follows that $T'$ has a $U(n')$-weighted zero-sum subsequence. As $n$ is squarefree, $p$ is coprime to $n'$. Thus, by Lemmas \ref{u2s} and \ref{lifts'} we see that $S$ has an $S(n)$-weighted zero-sum subsequence. 

\begin{case}
For each prime divisor $p$ of $n$, exactly two terms of $S$ are coprime to $p$.
\end{case}

Suppose $S$ has at most one unit. By Lemma \ref{gs}, we see that $S$ is an $S(n)$-weighted zero-sum sequence. So we can assume that $S$ has at least two units. By the assumption in this subcase, we see that $S$ will have exactly two units and the other terms of $S$ will be zero. As $S$ has length $k+1$ and as $k\geq 2$, some term of $S$ is zero. 

\begin{case}
For every prime divisor $p$ of $n$ at least two terms of $S$ are coprime to $p$, and there is a prime divisor $p'$ of $n$ such that at least three terms of $S$ are coprime to $p'$.
\end{case}

In this case, we are done by Lemma \ref{gs'}.
\end{proof}

\begin{thm}\label{csn}
Let $n$ be squarefree. If $n$ is a prime, we have $C_{S(n)}(n)=3$. If $n$ is not a prime and every prime divisor of $n$ is at least seven, we have $C_{S(n)}(n)=2^{\Omega(n)}$. 
\end{thm}

\begin{proof}
If $n=p$ where $p$ is a prime then $S(n)=Q_p$. As $p$ is odd, from Theorem \ref{q} we get that $C_{S(n)}(n)=3$. Let $n=p_1\ldots p_k$ where $k\geq 2$. As $S(n)\su U(n)$, it follows that $C_{S(n)}(n)\geq C_{U(n)}(n)$. As $n$ is odd, from Theorem \ref{un} we have $C_{S(n)}(n)\geq 2^k$. 

Let $S=(x_1,\ldots,x_l)$ be a sequence in $\Z_n$ of length $l=2^k$. If we show that $S$ has an $S(n)$-weighted zero-sum subsequence of consecutive terms, it will follow that $C_{S(n)}(n)\leq 2^k$. If any term of $S$ is zero, we get an $S(n)$-weighted zero-sum subsequence of $S$ of length 1. 

\begin{case} 
There is a prime divisor $p$ of $n$ such that at most one term of $S$ is coprime to $p$.
\end{case}

We will get a subsequence, say $T$, of consecutive terms of $S$ of length $l/2$ whose all terms are divisible by $p$. Let $n'=n/p$ and let $T'$ be the image of $T$ under $f_{n|n'}$. From Theorem \ref{un}, we have $C_{U(n')}(n')=2^{\Omega(n')}$. As the length of $T'$ is $2^{\Omega(n')}$, it follows that $T'$ has a $U(n')$-weighted zero-sum subsequence of consecutive terms. As $n'$ is coprime with $p$, by Lemmas \ref{u2s} and \ref{lifts'} we get that $T$ (and hence $S$) has an $S(n)$-weighted zero-sum subsequence of consecutive terms.

\begin{case}
For each prime divisor $p$ of $n$, exactly two terms of $S$ are coprime to $p$.
\end{case}

In this case, as $\Omega(n)=k$ there are at most $2k$ non-zero terms in $S$. Let $k\geq 3$. As $S$ has length $2^k$ and as $2^k>2k$, some term of $S$ is zero and we are done. 
If $k=2$, then $S$ has length 4. If $S$ has at most one unit, by Lemma \ref{gs} this sequence is an $S(n)$-weighted zero-sum sequence. So we can assume that $S$ has at least two units. By the assumption in this subcase we see that $S$ has exactly two units and so the other two terms of $S$ will be zero.   

\begin{case}
For every prime divisor $p$ of $n$ at least two terms of $S$ are coprime to $p$, and there is a prime divisor $p'$ of $n$ such that at least three terms of $S$ are coprime to $p'$.
\end{case}

In this case, we are done by Lemma \ref{gs'}.
\end{proof}

\section{Some results about the weight-set $L(n;p)$}\label{l}

To determine the constant $D_{S(n)}(n)$ for some non-squarefree $n$, we consider the following subset of $\Z_n$ as a weight-set. 

\begin{defn}
Let $p$ be a prime divisor of $n$ where $n$ is odd. We define 
$$L(n;p)=\Big\{\,a\in U(n)\,:\,\Big(\dfrac{a}{n}\Big) =\Big(\dfrac{a}{p}\Big)\,\Big\}$$  
\end{defn}

Consider the homomorphism $\varphi:U(n)\to\{1,-1\}$ given by $\varphi(a)=\Big(\dfrac{a}{n}\Big)\Big(\dfrac{a}{p}\Big)$. Then the kernel of $\varphi$ is $L(n;p)$ and so it follows that  $L(n;p)$ is a subgroup having index at most two in $U(n)$. 

\begin{prop}
Let $p$ be a prime divisor of $n$. Then $L(n;p)$ has index two in $U(n)$, unless $p$ is the unique prime divisor of $n$ such that $v_{p}(n)$ is odd.
\end{prop}
 
\begin{proof}
Let $n=p^rm$ where $m$ is coprime to $p$. Let $\psi:U(n)\to U(p^r)\times U(m)$ be the isomorphism which is given by the Chinese remainder theorem. If we show that $-1$ is in the image of the homomorphism $\varphi:U(n)\to\{1,-1\}$ which was defined above, then $ker \,\varphi$ will be a subgroup of index two in $U(n)$.

\begin{case}
$r$ is odd.
\end{case}

Suppose $m$ is a square. For any $a\in U(n)$, we have $\varphi(a)=\Big(\dfrac{a}{m}\Big)\Big(\dfrac{a}{p^{r+1}}\Big)=1$. Thus $\varphi$ is the trivial map and so $L(n;p)=U(n)$. 

Suppose $m$ is not a square. By Proposition \ref{indexs} we see that $S(m)$ has index two in $U(m)$. For $c\in U(m)\setminus S(m)$, there exists $a\in U(n)$ such that $\psi(a)=(1,c)$. Thus $\Big(\dfrac{a}{p}\Big)=\Big(\dfrac{1}{p}\Big)=1$ and so $\varphi(a)=\Big(\dfrac{a}{n}\Big)=\Big(\dfrac{a}{m}\Big)=-1$.

\bigskip

\begin{case}
$r$ is even.
\end{case}

Let $m=1$. Then $\Big(\dfrac{a}{n}\Big)=\Big(\dfrac{a}{p}\Big)^r=1$ and so $\varphi(a)=\Big(\dfrac{a}{p}\Big)$. Let $b\in U(p)\setminus Q_p$. There exists $a\in U(n)$ such that $f_{n|p}(a)=b$. Thus $\varphi(a)=\Big(\dfrac{b}{p}\Big)=-1$. 

Suppose $m>1$. Let $b\in U(p)\setminus Q_p$. There exists $b'\in U(p^r)$ such that $f_{p^r|p}(b')=b$. For $c\in S(m)$, there exists $a\in U(n)$ such that $\psi(a)=(b',c)$. Thus $\Big(\dfrac{a}{n}\Big)=\Big(\dfrac{b}{p}\Big)^r\Big(\dfrac{c}{m}\Big)=1$ and so $\varphi(a)=\Big(\dfrac{a}{p}\Big)=\Big(\dfrac{b}{p}\Big)=-1$.
\end{proof}

\smallskip

\begin{rem}
In particular if $n$ is a prime $p$, then $L(n;p)=U(p)$.
\end{rem}

\begin{lem}\label{s2l}
Let $p'$ be a prime divisor of $n$ and $p$ be a prime divisor of $n$ which is coprime with $n'=n/p$. Then $S(n')\su f_{n|n'}\big(L(n;p')\big)$. 
\end{lem}

\begin{proof}
Let $b\in S(n')$ where $n'=n/p$. As $p$ is coprime with $n'$, by the Chinese remainder theorem we have an isomorphism $\psi:U(n)\to U(n')\times U(p)$. 

Suppose $p=p'$. Let $a\in U(n)$ such that $\psi(a)=(b,1)$. Thus $f_{n|n'}(a)=b$ and $a\in L(n;p')$ as $$\Big(\dfrac{a}{n}\Big)=\Big(\dfrac{b}{n'}\Big)\Big(\dfrac{1}{p}\Big)=\Big(\dfrac{1}{p}\Big)=\Big(\dfrac{a}{p}\Big)=\Big(\dfrac{a}{p'}\Big).$$

Suppose $p\neq p'$. Then $p'$ divides $n'$. Let $c\in U(p)$ such that $\Big(\dfrac{c}{p}\Big)=\Big(\dfrac{b}{p'}\Big)$ and let $a\in U(n)$ such that $\psi(a)=(b,c)$. Thus $f_{n|n'}(a)=b$ \\
and $a\in L(n;p')$ as 
$$\Big(\dfrac{a}{n}\Big)=\Big(\dfrac{b}{n'}\Big)\Big(\dfrac{c}{p}\Big)=\Big(\dfrac{c}{p}\Big)=\Big(\dfrac{b}{p'}\Big)=\Big(\dfrac{a}{p'}\Big).
\eqno\qed$$
\renewcommand{\qedsymbol}{}
\vspace{-\baselineskip}
\end{proof}

\begin{lem}\label{u2l}
Let $p'$ be a prime divisor of $n$ which is coprime to $n'=n/p'$. Then $U(p')\su f_{n|p'}\big(L(n;p')\big)$.  
\end{lem}

\begin{proof}
Let $b\in U(p')$. As $n'=n/p'$ is coprime to $p'$, by the Chinese remainder theorem we have an isomorphism $\psi:U(n)\to U(n')\times U(p')$. There exists $a\in U(n)$ such that $\psi(a)=(1,b)$. Thus $f_{n|p'}(a)=b$ and $a\in L(n;p')$ as $$\Big(\dfrac{a}{n}\Big)=\Big(\dfrac{1}{n'}\Big)\Big(\dfrac{b}{p'}\Big)=\Big(\dfrac{b}{p'}\Big)=\Big(\dfrac{a}{p'}\Big).\eqno\qed$$
\renewcommand{\qedsymbol}{}
\vspace{-\baselineskip}
\end{proof}

\begin{obs}\label{obs3}
Let $A\su \Z_n$, $S$ be a sequence in $\Z_n$ and $n=m_1m_2$ where $m_1$ and $m_2$ are coprime. For $i=1,2$, let $A_i\su\Z_{m_i}$ be given and $S_i$ denote the image of the sequence $S$ under $f_{n|m_i}$. Suppose $A_1\times A_2\su \psi(A)$ where $\psi:U(n)\to U(m_1)\times U(m_2)$ is the isomorphism given by the Chinese remainder theorem. If $S_1$ is an $A_1$-weighted zero-sum sequence in $\Z_{m_1}$ and $S_2$ is an $A_2$-weighted zero-sum sequence in $\Z_{m_2}$, then $S$ is an $A$-weighted zero-sum sequence in $\Z_n$. 
\end{obs}

\begin{lem}\label{gl'}
Let $n$ be squarefree and $p'$ be a prime divisor of $n$. Suppose $n'=n/p'$ and $\psi:U(n)\to U(n')\times U(p')$ is the isomorphism given by the Chinese remainder theorem. Then $S(n')\times U(p')\su \psi\big(L(n;p')\big)$.
\end{lem}

\begin{proof}
Let $(b,c)\in S(n')\times U(p')$. There exists $a\in U(n)$ such that $\psi(a)=(b,c)$. Then $a\in L(n;p')$ as $$\Big(\dfrac{a}{n}\Big)=\Big(\dfrac{b}{n'}\Big)\Big(\dfrac{c}{p'}\Big)=\Big(\dfrac{c}{p'}\Big)=\Big(\dfrac{a}{p'}\Big).\eqno\qed$$
\renewcommand{\qedsymbol}{×}
\vspace{-\baselineskip}
\end{proof}

\section{The constants $D_{L(n;p)}(n)$ and $C_{L(n;p)}(n)$}

\begin{lem}\label{gl}
Let $n$ be squarefree, $p'$ be a prime divisor of $n$, $S=(x_1,\ldots,x_l)$ be a sequence in $\Z_n$ such that for every prime divisor $p$ of $n$ at least two terms of $S$ are coprime to $p$. Assume that $S'$ denotes the image of $S$ under $f_{n|n'}$, where $n'=n/p'$. Suppose at most one term of $S'$ is a unit or suppose there is a prime divisor $p$ of $n/p'$ such that at least three terms of $S$ are coprime to $p$. Then $S$ is an $L(n;p')$-weighted zero-sum sequence. 
\end{lem}

\begin{proof}
Let $n'=n/p'$ and $S'$ denote the image of the sequence $S$ under $f_{n|n'}$.  As at least two terms of $S^{(p')}$ are coprime to $p'$, by Lemma \ref{gri} we have $S^{(p')}$ is a $U(p')$-weighted zero-sum sequence.

If at most one term of $S'$ is a unit, by Lemma \ref{gs} we see that $S'$ is an $S(n')$-weighted zero-sum sequence in $\Z_{n'}$, as $n'$ is squarefree and for every prime divisor $p$ of $n'$ at least two terms of $S'$ are coprime to $p$. 

If there is a prime divisor $p$ of $n/p'$ such that at least three terms of $S$ are coprime to $p$, by Lemma \ref{gs'} we get that $S'$ is an $S(n')$-weighted zero-sum sequence, since at least three terms of $S'$ are coprime to $p$. 

As $n$ is squarefree, $n'$ is coprime to $p'$. Let $\psi:U(n)\to U(n')\times U(p')$ be the isomorphism given by the Chinese remainder theorem. By Lemma \ref{gl'} we see that $S(n')\times U(p')\su \psi\big(L(n;p')\big)$. Hence, by Observation \ref{obs3} we see that $S$ is an $L(n;p')$-weighted zero-sum sequence. 
\end{proof}

\begin{thm}\label{ld}
Let $n$ be a squarefree number whose every prime divisor is at least seven. Suppose that $p'$ is a prime divisor of $n$ and $\Omega(n)\neq 2$. Then $D_{L(n;p')}(n)=\Omega(n)+1$. 
\end{thm}

\begin{proof}
Let $p'$ be a prime divisor of $n$. We have $D_{U(n)}(n)\leq D_{L(n;p')}(n)$, as $L(n;p')\su U(n)$. From Theorem \ref{un} we have $D_{U(n)}(n)=\Omega(n)+1$ and so $D_{L(n;p')}(n)\geq \Omega(n)+1$. If $\Omega(n)=1$, then $L(n;p')=U(n)$ and so by Theorem \ref{un} we have $D_{L(n;p')}(n)=2$. 

Let $n$ be a squarefree number whose every prime divisor is at least seven. Suppose $\Omega(n)\geq 3$ and $S=(x_1,\ldots,x_l)$ is a sequence in $\Z_n$ of length $\Omega(n)+1$. To show that $D_{L(n;p')}(n)\leq \Omega(n)+1$, it suffices to show that $S$ has an $L(n;p')$-weighted zero-sum subsequence. 

\begin{case}
There is a prime divisor $p$ of $n$ such that at most one term of $S$ is coprime to $p$.
\end{case}

Let us assume without loss of generality that $x_i$ is divisible by $p$ for $i>1$ and let $T$ denote the subsequence $(x_2,\ldots,x_l)$ of $S$. Let $n'=n/p$ and let $T'$ denote the sequence in $\Z_{n'}$ which is the image of $T$ under $f_{n|n'}$. We have $n'$ is squarefree, $\Omega(n')\geq 2$, every prime divisor of $n'$ is at least seven and $T'$ has length $\Omega(n')+1$. 

So it follows from Theorem \ref{dsn} that $T'$ has an $S(n')$-weighted zero-sum subsequence. As $n$ is squarefree, $p$ is coprime to $n'$. Now by Lemmas \ref{lifts'} and \ref{s2l}, we see that $T$ has an $L(n;p')$-weighted zero-sum subsequence.

\begin{case}
For every prime divisor $p$ of $n/p'$ exactly two terms of $S$ are coprime to $p$, and at least two terms of $S$ are coprime to $p'$.
\end{case} 

Let $n'=n/p'$ and $S'=(x_1',\ldots,x_l')$ be the image of  $S$ under $f_{n|n'}$. Suppose at most one term of $S'$ is a unit. By Lemma \ref{gl} we see that $S$ is an $L(n;p')$-weighted zero-sum sequence. Suppose at least two terms of $S'$ are units. By the assumption in this case we see that exactly two terms of $S'$ are units, say $x_{j_1}'$ and $x_{j_2}'$ and the other terms of $S'$ are zero. It follows that all terms of $S$ are divisible by $n'$ except $x_{j_1}$ and $x_{j_2}$.  

Hence, if some term $f_{n|p'}(x_j)$ of $S^{(p')}$ is zero for $j\neq j_1,j_2$, then $x_j=0$. So we can assume that all the terms of $S^{(p')}$ are non-zero except possibly two terms. As $k\geq 3$, the sequence $S$ has length at least 4. Let $T$ be a subsequence of $S$ of length at least two which does not contain the terms $x_{j_1}$ and $x_{j_2}$. 

As all the terms of $T^{(p')}$ are non-zero and as $T^{(p')}$ has length at least 2, by Lemma \ref{gri} we see that $T^{(p')}$ is a $U(p')$-weighted zero-sum sequence. Also all the terms of $T$ are divisible by $n'$.  Hence, by Lemmas \ref{lifts'} and \ref{u2l} we see that $T$ is an $L(n;p')$-weighted zero-sum subsequence of $S$.

\begin{case}
Given any prime divisor $p$ of $n$ at least two terms of $S$ are coprime to $p$, and there is a prime divisor $p$ of $n/p'$ such that at least three terms of $S$ are coprime to $p$.
\end{case}

In this case, we are done by Lemma \ref{gl}.
\end{proof}

\begin{thm}\label{ld2}
Let $n=p'q$ where $p'$ and $q$ are distinct primes which are at least seven. Then $D_{L(n;p')}(n)=4$.
\end{thm}

\begin{proof}
Let $n$ be as in the statement of the theorem. As $L(n;p')\su U(n)$, we have $f_{n|p'}\big(L(n;p')\big)\su U(p')$. Also observe that $f_{n|q}\big(L(n;p')\big)\su Q_q$. As from Theorem \ref{un} we have $D_{U(p')}(p')=2$ and from Theorem \ref{q} we have $D_{Q_q}(q)=3$, by Lemma \ref{dadd} it follows that $D_{L(n;p')}(n)\geq 4$. 

Let $S=(x_1,x_2,x_3,x_4)$ be a sequence in $\Z_n$. We will show that $S$ has an $L(n;p')$-weighted zero-sum subsequence. It will follow that $D_{L(n;p')}(n)=4$. If some term of $S$ is zero, then we are done. So we can assume that all the terms of $S$ are non-zero.   We continue with the notations and terminology which were used in the proof of Theorem \ref{ld}. 

\begin{case}
There is a prime divisor $p$ of $n$ such that at most one term of $S$ is coprime to $p$.
\end{case}

We can find a subsequence $T$ of $S$ of length 3 such that all the terms of $T$ are divisible by $p$. Let $n'=n/p$ and let $T'$ be the sequence in $\Z_{n'}$ which is the image of $T$ under $f_{n|n'}$. As all the terms of $S$ are non-zero, no term of $T$ can be divisible by $n'$. So $T'$ is a sequence of non-zero terms of length 3. As $n'$ is a prime, $S(n')=Q_{n'}$ and by Corollary \ref{cm} we see that $T'$ is a $Q_{n'}$-weighted zero-sum subsequence. Thus, by Lemmas \ref{lifts'} and \ref{s2l} we see that $T$ is an $L(n;p')$-weighted zero-sum subsequence of $S$.

\begin{case}
Exactly two terms of $S$ are coprime to $q$.
\end{case}

Let us assume that $x_1$ and $x_2$ are coprime to $q$ and let $T:(x_3,x_4)$. The sequence $T^{(q)}$ has both terms zero and hence it is an $S(q)$-weighted zero-sum sequence. As $S$ has all terms non-zero, we see that both the terms of $T^{(p')}$ are non-zero, and so by Lemma \ref{gri} we get that $T^{(p')}$ is a $U(p')$-weighted zero-sum sequence. Let $\psi:U(n)\to U(q)\times U(p')$ be the isomorphism given by the Chinese remainder theorem. By Lemma \ref{gl'} we have $S(q)\times U(p')\su \psi\big(L(n;p')\big)$. Thus, by Observation \ref{obs3} we see that $T$ is an $L(n;p')$-weighed zero-sum subsequence of $S$. 

\begin{case}
At least three terms of $S$ are coprime to $q$, and at least two terms of $S$ are coprime to $p'$.
\end{case}

In this case, we are done by Lemma \ref{gl}.
\end{proof}

\begin{thm}\label{lc}
Let $n$ be squarefree whose every prime divisor is at least seven. Suppose $p'$ is a prime divisor of $n$ and $\Omega(n)\neq 2$. Then $C_{L(n;p')}(n)=2^{\Omega(n)}$. 
\end{thm}

\begin{proof}
If $n$ is a prime, then $n=p'$ and $L(n;p')=U(p')$. So from Theorem \ref{un} we have  $C_{L(n;p')}(n)=2$. Let $n=p_1\ldots p_k$ where $k\geq 3$ and let $p'=p_k$. As $L(n;p')\su U(n)$,  we have $C_{L(n;p')}(n)\geq C_{U(n)}(n)$. So from Theorem \ref{un}, we have $C_{L(n;p')}(n)\geq 2^{\Omega(n)}$. Let $S=(x_1,\ldots,x_l)$ be a sequence in $\Z_n$ of length $l=2^{\Omega(n)}$. If we show that $S$ has an $L(n;p')$-weighted zero-sum subsequence of consecutive terms, it will follow that $C_{L(n;p')}(n)\leq 2^{\Omega(n)}$. If any term of $S$ is zero, then we get an $L(n;p')$-weighted zero-sum subsequence of $S$ of length 1. 

\begin{case}
There is a prime divisor $p$ of $n$ such that at most one term of $S$ is coprime to $p$.
\end{case}

We can find a subsequence say $T$ of consecutive terms of $S$ of length $l/2$ such that all the terms of $T$ are divisible by $p$. Let $n'=n/p$ and let $T'$ be the image of $T$ under $f_{n|n'}$. As $\Omega(n')=\Omega(n)-1\geq 2$ and as $T'$ has length $2^{\Omega(n')}$, by Theorem \ref{csn} we see that $T'$ has an $S(n')$-weighted zero-sum subsequence of consecutive terms. By Lemma \ref{s2l} we get $S(n')\su f_{n|n'}\big(L(n;p')\big)$ and so by Lemma \ref{lifts'} we get that $T$ (and hence $S$) has an $L(n;p')$-weighted zero-sum subsequence of consecutive terms. 

\begin{case}
For every prime divisor $p$ of $n/p'$ exactly two terms of $S$ are coprime to $p$, and at least two terms of $S$ are coprime to $p'$.
\end{case}

In this case, we can use a slight modification of the argument which was used in the same case of the proof of Theorem \ref{ld}. We just observe that in a sequence $S$ of length at least eight which has at most two terms which are not divisible by $n'$, we can find a subsequence $T$ of consecutive terms of length at least two such that all the terms of $T$ are divisible by $n'$. 

\begin{case}
For every prime divisor $p$ of $n$ at least two terms of $S$ are coprime to $p$, and there is a prime divisor $p$ of $n/p'$ such that at least three terms of $S$ are coprime to $p$.
\end{case}

In this case, we are done by Lemma \ref{gl}.
\end{proof}

\begin{thm}\label{lc2}
Let $n=p'q$ where $p'$ and $q$ are distinct primes which are at least seven. Then $C_{L(n;p')}(n)=6$.
\end{thm}

\begin{proof}
Let $n$ be as in the statement of the theorem. By Theorems \ref{un} and \ref{q}, we see that $C_{U(p')}(p')=2$ and $C_{Q_q}(q)=3$. Also as $f_{n|p'}\big(L(n;p')\big)\su U(p')$ and $f_{n|q}\big(L(n;p')\big)\su Q_q$, by Lemma \ref{clb} it follows that $C_{L(n;p')}(n)\geq 6$.

Let $S=(x_1,\ldots,x_6)$ be a sequence in $\Z_n$. It is enough to show that $S$ has an $L(n;p')$-weighted zero-sum subsequence of consecutive terms. We can assume that all the terms of $S$ are non-zero. 

\begin{case}
There is a prime divisor $p$ of $n$ such that at most one term of $S$ is coprime to $p$.
\end{case}

In this case, we can find a subsequence $T$ of $S$ of consecutive terms of length three whose all terms are divisible by $p$. As all the terms of $S$ are non-zero, all the terms of $T$ are coprime to $n'$ where $n'=n/p$. If $T'$ is the image of $T$ under $f_{n|n'}$, then $T'$ is a sequence of non-zero terms of length three in $\Z_{n'}$ . As $n'$ is a prime, $S(n')=Q_{n'}$ and by Corollary \ref{cm} we get that $T'$ is a $Q_{n'}$-weighted zero-sum sequence. By using Lemmas \ref{lifts'} and \ref{s2l} it follows that $T$ is an $L(n;p')$-weighted zero-sum subsequence of $S$ of consecutive terms.

\begin{case}
Exactly two terms of $S$ are coprime to $q$.
\end{case}

Let the terms $x_{j_1}$ and $x_{j_2}$ be coprime to $q$. As $S$ has length six, we can find a subsequence $T$ of consecutive terms of $S$ of length two, which does not have any term from the positions $j_1$ and $j_2$. As $x_j$ is divisible by $q$ when $j\neq j_1,j_2$, all the terms of $T$ are divisible by $q$. As $S$ has all terms non-zero, all the terms of $T$ are coprime to $p'$. 

By Lemma \ref{gri} we get that $T^{(p')}$ is a $U(p')$-weighted zero-sum sequence. So by Lemmas \ref{lifts'} and \ref{u2l} it follows that $T$ is an $L(n;p')$-weighted zero-sum subsequence of consecutive terms of $S$. 

\begin{case}
At least three terms of $S$ are coprime to $q$, and at least two terms of $S$ are coprime to $p'$.
\end{case}

In this case, we are done by Lemma \ref{gl}.
\end{proof}

\section{Concluding remarks}

We have $S(15)=\{1,2,4,8\}$. We can check that the sequence $S:(1,1,1)$ does not have any $S(15)$-weighted zero-sum subsequence. So $D_{S(15)}(15)\geq 4$ and hence $D_{S(15)}(15)>\Omega(15)+1$. This shows that the statement of Theorem \ref{dsn} is not true in general if some prime divisor of $n$ is smaller than seven. It will be interesting to find the Davenport constant $D_{S(n)}(n)$ for non-squarefree $n$. 

In \cite{ACFKP}, it was proposed to characterize the weight-sets $A\su\Z_n$ which have the same value of $D_A(n)$. In this paper we have seen that if $A\su \Z_n$ is such that $S(n)\su A\su U(n)$ and if $n$ is not a prime, then $D_A(n)=D_{U(n)}(n)$. We have also seen that if $A\su\Z_n$ is such that $L(n;p)\su A\su U(n)$ and if $\Omega(n)\neq 2$, then again $D_A(n)=D_{U(n)}(n)$. We can try to see whether this can happen for some other weight-sets $A\su\Z_n$.

\bigskip

{\bf Acknowledgement.}
Santanu Mondal would like to acknowledge CSIR, Govt. of India, for a research fellowship.

\end{document}